\documentclass[12pt]{article}
\usepackage[margin=3cm]{geometry}
\usepackage[utf8]{inputenc}
\usepackage{float}
  \usepackage{amsmath}
  \usepackage{amsthm}
 \usepackage{amssymb}
\usepackage{amsfonts}
\usepackage{graphics}
\usepackage{hyperref}
 \usepackage[english]{babel}
 \newlength\tindent
\newtheorem{theorem}{Theorem}
\newtheorem*{theoremm}{Theorem}
\newtheorem{lemma}{Lemma}
\newtheorem{corollary}{Corollary}
\newtheorem*{definition}{Definition}

\newtheorem{remark}{Remark}

\providecommand{\keywords}[1]
{
  \small	
  \textbf{Keywords:} #1
}
\providecommand{\subjclass}[1]
{
  \small	
  \textbf{MSC Class:} #1
}

\begin{document}

\title{A polynomial analogue of Jacobsthal function}
\author{Alexander Kalmynin \footnote{National Research University Higher School of Economics, Moscow, Russia} \footnote{Steklov Mathematical Institute of Russian Academy of Sciences, Moscow, Russia} \\ email: \href{mailto:alkalb1995cd@mail.ru}{alkalb1995cd@mail.ru} 
\and Sergei Konyagin \footnotemark[\value{footnote}] \\ email: \href{mailto:konyagin23@gmail.com}{konyagin23@gmail.com}}
\date{}
\maketitle
\begin{abstract}
For a polynomial $f(x)\in \mathbb Z[x]$ we study an analogue of Jacobsthal function, defined by the formula
\[
j_f(N)=\max_{m}\{\text{For some } x\in \mathbb N \text{ the inequality } (x+f(i),N)>1 \text{ holds for all }i\leq m\}.
\]
We prove a lower bound
\[
j_f(P(y))\gg y(\ln y)^{\ell_f-1}\left(\frac{(\ln\ln y)^2}{\ln\ln\ln y}\right)^{h_f}\left(\frac{\ln y\ln\ln\ln y}{(\ln\ln y)^2}\right)^{M(f)},
\]
where $P(y)$ is the product of all primes $p$ below $y$, $\ell_f$ is the number of distinct linear factors of $f(x)$, $h_f$ is the number of distinct non-linear irreducible factors and $M(f)$ is the average size of the maximal preimage of a point under a map $f:\mathbb F_p\to \mathbb F_p$. The quantity $M(f)$ is computed in terms of certain Galois groups.
\end{abstract}


\section{Introduction and main results}

Let $N$ be a natural number. The Jacobsthal function $j(N)$ is defined as the length of the largest gap between numbers which are coprime to $N$. In other words,
\[
j(N)=\max_{m}\{\text{For some } x\in \mathbb N \text{ the inequality } (x+i,N)>1 \text{ holds for all }i\leq m\}
\]
Estimates for $j(N)$ are used to establish results on large gaps between consecutive primes. For instance, the current best lower bound for $p_{n+1}-p_n$ is
\begin{theoremm}[\cite{FGKMT}]
Let $y\geq 19$ and
\[
P(y)=\prod_{p\leq y}p.
\]
Then
\[
j(P(y))\gg y\frac{\ln y\ln\ln\ln y}{\ln\ln y}=(1+o_{y\to \infty}(1))\ln P(y)\frac{\ln\ln P(y)\ln\ln\ln\ln P(y)}{\ln\ln\ln P(y)}.
\]
In particular, for $X\geq e^{16}$
\[
\max_{p_{n+1}\leq X}(p_{n+1}-p_n)\gg \ln X\frac{\ln\ln X\ln\ln\ln\ln X}{\ln\ln\ln X}.
\]
\end{theoremm}

As for the upper bounds for $j(N)$, H. Iwaniec showed \cite{Iw} that
\begin{theoremm}
For all natural $N$ the inequality
\[
j(N)\ll \ln^2 N
\]
holds.
\end{theoremm}

In this paper, we are going to study a polynomial analogue of Jacobsthal function. Namely, let $f(x)$ be a non-constant polynomial with integer coefficients. Define the function $j_f(x)$ by the formula
\[
j_f(N)=\max_{m}\{\text{For some } x\in \mathbb N \text{ the inequality } (x+f(i),N)>1 \text{ holds for all }i\leq m\}.
\]
So, instead of the intervals not containing numbers coprime to $N$ we study shifts of polynomial sequences. Here we are going to present a lower bound for $j_f(P(y))$ similar to the classical Rankin's bound \cite{Ra}. To formulate the resulting statement, let us first define a number $M(f)$, which one can call an average size of the maximal preimage of a non-zero point under a map $f:\mathbb F_p\to \mathbb F_p$.
\begin{definition}
Suppose that $f(x)\in \mathbb Z[x]$. If $p$ is a prime number, denote by $M_p(f)$ the maximal number of solutions $x\in \mathbb F_p$ to the equation $f(x)=y$ for fixed $y\in \mathbb F_p\backslash\{0\}$. The quantity $M(f)$ is defined as the unique $M$ with
\[
\sum_{p\leq X}\frac{M_p(f)}{p}=M\ln\ln X+O(1)
\]
for all real $X\geq 2$, if such a number $M$ exists.
\end{definition}

It is easy to see that if $f_d(x)=x^d$, then $M(f_1)=1$ and $M(f_2)=M(f_3)=2$. In the third section we will express $M(f)$ in terms of certain Galois groups associated with $f(x)$. We compute $M(f_d)$ explicitly for all $d$ and also for a typical polynomial of degree $d$. The second section is devoted to the proof of the following bound for $j_f(P(y))$
\begin{theorem}
Suppose that for a polynomial $f\in \mathbb Z[x]$ the factorization of $f$ contains $\ell_f$ distinct linear and $h_f$ distinct non-linear factors. Then for $y\geq 19$ the inequality
\[
j_f(P(y))\gg y(\ln y)^{\ell_f-1}\left(\frac{(\ln\ln y)^2}{\ln\ln\ln y}\right)^{h_f}\left(\frac{\ln y\ln\ln\ln y}{(\ln\ln y)^2}\right)^{M(f)}=
\]
\[
(1+o_{y\to\infty}(1))\ln P(y)(\ln\ln P(y))^{\ell_f-1}\left(\frac{(\ln\ln\ln P(y))^2}{\ln\ln\ln\ln P(y)}\right)^{h_f}\left(\frac{\ln\ln P(y)\ln\ln\ln\ln P(y)}{(\ln\ln\ln P(y))^2}\right)^{M(f)}
\]
holds.
\end{theorem}
\begin{remark}
Let us observe that the value of $j_f$ is invariant under shifts by a constant, i.e. for all $n\in \mathbb Z$ we have $j_f=j_{f+n}$. Choosing $n=-f(0)$, we can always force our polynomial to be divisible by $x$. In particular, for a polynomial $f$ of degree at least $2$ we will always get
\[
j_f(P(y))\gg y\frac{(\ln\ln y)^2}{\ln\ln\ln y}\frac{\ln y\ln\ln\ln y}{(\ln\ln y)^2}=y\ln y.
\]
Indeed, since $M(f)\geq 1$, for $\ell_f\geq 2$ or $\ell_f=1$ and $h_f\geq 1$ we get the desired inequality. The last remaining case is $\ell_f=1, h_f=0$, i.e. $f=c(ax+b)^d$ for some $a,b,c\in \mathbb Z$, $ac\neq 0$. In the last section we show that $M(f_d)=\tau(d)\geq 2$ for $f_d(x)=x^d$, $d\geq 2$, hence in our case we get $M(f)=M(f_d)\geq 2$ and the inequality follows.
\end{remark}

It turns out that the value $M(f)$ always exists and is rational. To formulate our result on $M(f)$, let us first introduce the quantity $m(G,H;X)$.
\begin{definition}
Let $G$ be a finite group acting on a set $X$ and $H$ be a subgroup of $G$. For $g\in G$ denote by $X^g$ the set of fixed points of $g$. Then
\[
m(G,H;X)=\frac{1}{|G|}\sum_{g\in G}\max_{h\in H}|X^{hg}|.
\]
\end{definition}
\begin{theorem}
For a field $L$, let $L_f$ denote the splitting field of the polynomial $f(x)-y$ over the field $L(y)$ of rational functions over $L$. Denote by $K$ the intersection of $\mathbb Q_f$ and $\overline{\mathbb Q}$, i.e. the maximal constant subfield of $\mathbb Q_f$. Let $G_f$ and $G_f^+$ be the Galois groups of $\mathbb Q_f=K_f$ over $\mathbb Q(y)$ and $K(y)$, respectively. These groups have a natural action on the set $X$ of all solutions to $f(x)=y$ and $G_f^+$ may be viewed as a subgroup of $G_f$ fixing $K$. We have
\[
M(f)=m(G_f,G_f^+;X).
\]
\end{theorem}

Our study of $j_f$ is motivated by Theorem 5 of the work \cite{DEKKM}. This result gives an upper bound for the least integer $\gamma_k>0$ such that all the numbers $\gamma_k+j^d$ for $1\leq j\leq k$ are not sums of two squares. The problem of estimating $\gamma_k$ was raised by the first two authors.

\section{Proof of the main theorem}

To prove our main result, we are going to choose a value $x$ such that $x+f(i)$ is not coprime to $P(y)$ for many first values of $i$ by choosing residues $x_p \mod p$ for all $p\leq y$. We will choose $x_p \mod p$ for primes $p\le y$ so that for all $i=1,\ldots,m$ there is some $p\le y$ for which $f(i)+x_p \equiv 0\pmod{p}$. Then we take $x$ to be a solution of the simultaneous congruences $x\equiv x_p\pmod{p}$ for $p\le y$. It then follows that $j_f(P(y)) \ge m$. In the proof we will use the following well-known result on sifted sets
\begin{lemma}
Suppose that $\kappa>0, z\geq 2$. Assume that $\{a_n\}$ is a sequence of non-negative real numbers such that for all $d \mid P(z)$ we have
\[
\sum_{n\equiv 0\pmod d}a_n=g(d)X/d+r_d,
\]
where $g(d)$ is a multiplicative function with $g(p)\leq \kappa, g(p)<p$ for all primes $p$, $|r_d|\leq g(d)$ and $z\ll X$. Then we have
\[
\mathcal S(a,z)=\sum_{(n,P(z))=1}a_n\ll_{\kappa} XV(z),
\]
where
\[
V(z)=\prod_{p\leq z}\left(1-\frac{g(p)}{p}\right).
\]
\end{lemma}
\begin{proof}
This is a version of the fundamental lemma of sieve theory. See, for example, \cite[Theorem 2.2]{HaRi}
\end{proof}

In particular, we get the following corollary
\begin{corollary}
Let $\kappa,z, g(d),V(z)$ and $X$ be as above. Suppose that for any $p\leq z$ the set $\Omega_p\subset \mathbb Z\slash p\mathbb Z$ contains $g(p)$ elements. Let $S(X,\Omega)$ be the number of $n\leq X$ such that $n\mod p \not\in \Omega_p$ for all $p\leq z$. Then
\[
S(X,\Omega)\ll XV(z).
\]
\end{corollary}
\begin{proof}
To see this, set for any $p\leq z$
\[
P(z;p)=\frac{P(z)}{p}
\]
and let $Q(z;p)$ be the least positive integer solution to the congruences
\[
Q(z;p)\equiv P(z;p) \pmod p \text{ and }Q(z;p)\equiv -1 \pmod{P(z;p)}.
\]
Setting
\[
a_m=\begin{cases}
1 \text{ if }m=\prod\limits_{p\leq z}\prod\limits_{r \in \Omega_p}(P(z;p)n+rQ(z;p)) \text{ for some }n\leq X\\
0\text{ otherwise}
\end{cases}
\]
one can easily see that we have $(m,P(z))=1$ and $a_m=1$ if and only if the corresponding $n$ does not lie in any $\Omega_p$ for $p\leq z$. Conditions of Lemma 1 also clearly hold.  
\end{proof}
Also, we need the following well-known consequence of Chebotarev's density theorem
\begin{lemma}
Let $f$ be an irreducible polynomial with integer coefficients. If $r_p(f)$ is the number of roots of $f$ in $\mathbb F_p$, then there are two constants $c_1,c_2>0$ such that for $x\geq 3$
\[
\sum_{p\leq x}\frac{r_p(f)}{p}=\ln\ln x+c_1+O(e^{-c_2\sqrt{\ln x}}).
\]
\end{lemma}
\begin{proof}
Let $K$ be the splitting field of $f$ over $\mathbb Q$. Galois group $G$ of $K$ over $\mathbb Q$ acts on the set of roots of $f$, this action is transitive. It is easy to see that for large enough $p$ the number $r_p(f)$ is equal to the number of fixed points of this action for the associated Frobenius conjugacy class. Therefore, by Chebotarev's theorem, the proportion of primes $p$ for which $r_p(f)=k$ is equal to proportion of elements of $G$ with exactly $k$ fixed points. By Burnside's lemma, the average number of fixed points for a transitive action is equal to $1$, therefore the average of $r_p(f)$ over primes is also $1$. Asymptotic formulas with required remainder terms follow from the effective version of Chebotarev's density theorem for $K\slash \mathbb Q$, see \cite[Theorem 1.3]{LaOd}.
\end{proof}

Let us now prove Theorem 1.
\begin{proof}[Proof of Theorem 1]
Take a large constant $A$ and an even larger constant $B$ and set
\[
m=\frac{y}{B}(\ln y)^{\ell_f-1}\left(\frac{(\ln\ln y)^2}{\ln\ln\ln y}\right)^{h_f}\left(\frac{\ln y\ln\ln\ln y}{(\ln\ln y)^2}\right)^{M(f)},
\]
\[
z_0=(\ln y)^A, z_1=\exp\left(\frac{\ln\ln \ln y\ln y}{A\ln\ln y}\right).
\]
We are going to choose $x_p$ in three steps. After the first two steps, we compute the number of unsifted numbers $i\leq m$.  To be more precise, if we have already chosen $x_p$ for some values of $p$, then the number $i$ is called sifted if among used primes there is a value $p$ such that
\[
f(i)+x_p\equiv 0 \pmod p.
\]
For the first step, let us choose
\[
x_p=0 \text{ for }p\leq z_0\text{ and for }z_1<p<y/2.
\]
Now, suppose that the number $i\leq m$ is not sifted after the first step. Let $k(x)$ be a factor of $f(x)$ of degree $1$. Then $k(i)$ has no prime factors $p\in [2,z_0]\cup (z_1,y/2)$. Since $k(i)\ll m\ll y(\ln y)^{\ell_f+M(f)}$, for large enough $A$ we conclude that $|k(i)|$ is either a prime or a $z_1-$smooth number. For large enough $A$ the number of $z_1-$smooth numbers that are $\ll m$ can be estimated as
\[
\Psi(O(m),z_1)\ll \frac{m}{(\ln y)^{\ell_f+M(f)+2}}=o\left(\frac{y}{\ln y}\right).
\]
For the second step, let us choose $x_p$ for all $z_0<p\leq z_1$ as follows: by the definition of $M(f)$, for any such $p$ there is a non-zero $y_p\in \mathbb F_p\backslash \{0\}$ such that the congruence $f(i)\equiv y_p \pmod p$ has $M_p(f)$ solutions. We set $x_p\equiv -y_p \pmod p$.

Next, we are going to estimate the number $R$ of $i\leq m$ that remain unsifted after these two steps. Our goal will be to establish the inequality
\[
R\leq \frac{y}{3\ln y}
\]
for all large $y$, if $A$ and $B$ are taken to be large enough.
 For $p\leq \sqrt{y}$ let $\Omega_p^I$ consist of all residues $t \mod p$ such that $k(t)\equiv 0\pmod p$ for some linear factor $k(x)$ of $f$. Next, if $p\in [2,z_0]\cup (z_1,\sqrt{y})$ we set
\[
\Omega_p^{II}=\{t \mod p:\exists\text{ non-linear irreducible factor }q(x)\text{ of }f(x)\text{ with  }q(t)\equiv 0 \pmod p\}.
\]
For $p\not\in[2,z_0]\cup(z_1,\sqrt{y})$ we define $\Omega_p^{II}=\emptyset$. Finally, for $z_0<p\leq z_1$ we set
\[
\Omega_p^{III}=\{t\mod p: f(t)\equiv y_p\pmod p\}
\]
and $\Omega_p^{III}=\emptyset$ for other values of $p$. Set $\Omega_p=\Omega_p^{I}\cup \Omega_p^{II}\cup \Omega_p^{III}$. We claim that if $i\leq m$ is unsifted after two steps, then either $i\ll \sqrt{y}$ or $k(i)$ is $z_1-$smooth for some linear factor $k(x)$ of $f$ (the number of such $i$ is $o(y/\ln y)$ by the choice of $A$) or $i\mod p\not \in \Omega_p$ for any $p\leq \sqrt{y}$. Indeed, suppose that $i\mod p\in \Omega_p$ for some $p$. Then there are three possible cases.

Case 1: $k(i)$ is divisible by $p$ for some linear factor $k(x)$ of $f$. In this case, either $k(i)=\pm p$ and $i\ll \sqrt{y}$, or $|k(i)|$ is not a prime number, hence $k(i)$ must be $z_1-$smooth or else $i$ must be sifted after the first step. 

Case 2: $p\in [2,z_0]\cup (z_1,\sqrt{y})$ and for some irreducible non-linear factor $q(x)$ of $f(x)$ we have $q(i)\equiv 0\pmod p$. Then $f(i)$ is also divisible by $p$, hence $p$ must be sifted after the first step. 

Case 3: $z_0<p\leq z_1$ and $f(i)\equiv y_p$. Then $i$ is sifted after the second step.

From this we conclude that

\[
R\leq S(m,\Omega)+O(\sqrt{y}+\Psi(O(m),z_1))=S(m,\Omega)+o(y/\ln y)
\]
by our choice of $A$. We will now apply Corollary 1 to estimate $S(m,\Omega)$. Let $g(p)=|\Omega_p|$. If for some $p\leq \sqrt{y}$ we have $g(p)=p$, then clearly $S(m,\Omega)=0$. Otherwise notice that all three parts of $\Omega_p$ contain at most $d=\deg f$ elements, hence Corollary is applicable for $\kappa=3d$ and we get
\[
S(m,\Omega)\ll m\prod_{p\leq \sqrt{y}}\left(1-\frac{g(p)}{p}\right)\ll m\exp\left(-\sum_{p\leq \sqrt{y}}\frac{g(p)}{p}\right).
\]
Now, let $k_1(x),\ldots,k_{\ell_f}(x)$ and $q_1(x),\ldots,q_{h_f}(x)$ be all the linear and non-linear irreducible factors of $f(x)$, respectively. Any two fixed irreducible polynomials have no common roots modulo large enough primes. Therefore, there is a constant $p_0$ such that for $p>p_0$ the sets $\Omega_p^{I}, \Omega_p^{II}$ and $\Omega_p^{III}$ are pairwise disjoint. Moreover, taking $p_0$ large enough, we can assure that $|\Omega_p^{I}|=\ell_f$ for all $p>p_0$ and $|\Omega_p^{II}|=r_p(q_1)+\ldots+r_p(q_{h_f})$ for $p\in (p_0,z_0]\cup (z_1,\sqrt{y})$. Also, by definition we have $|\Omega_p^{III}|=M_p(f)$ for all $z_0<p\leq z_1$. Consequently,
\[
\sum_{p\leq \sqrt{y}}\frac{g(p)}{p}=\sum_{p_0<p\leq \sqrt{y}}\frac{\ell_f}{p}+\sum_{h\leq h_f}\sum_{p\in (p_0,z_0]\cup (z_1,\sqrt{y})}\frac{r_p(q_h)}{p}+\sum_{z_0<p\leq z_1}\frac{M_p(f)}{p}.
\]
By Mertens theorem, the first sum is
\[
\sum_{p_0<p\leq \sqrt{y}}\frac{\ell_f}{p}=\ell_f \ln\ln y+O(1).
\]
Lemma 2 gives for every $h\leq h_f$
\[
\sum_{p\in (p_0,z_0]\cup (z_1,\sqrt{y})}\frac{r_p(q_h)}{p}=\ln\ln y-\ln\ln z_1+\ln\ln z_0+O(1).
\]
Finally, by the definition of $M(f)$, the third sum can be evaluated as
\[
\sum_{z_0<p\leq z_1}\frac{M_p(f)}{p}=M(f)(\ln\ln z_1-\ln\ln z_0)+O(1).
\]
Plugging all of this into our estimate, we get
\[
S(m,\Omega)\ll m(\ln y)^{-\ell_f}\left(\frac{\ln z_1}{\ln y\ln z_0}\right)^{h_f}\left(\frac{\ln z_0}{\ln z_1}\right)^{M(f)}=
\]
\[
=m(\ln y)^{-\ell_f}\left(\frac{\ln\ln\ln y}{A^2(\ln\ln y)^2}\right)^{h_f}\left(\frac{A^2(\ln\ln y)^2}{\ln y\ln\ln\ln y}\right)^{M(f)}=\frac{A^{2M(f)-2h_f}y}{B\ln y}.
\]
Since the implied constant does not depend on $A$ and $B$, we can take $B$ large enough to obtain
\[
S(m,\Omega)\leq \frac{y}{4\ln y},
\]
so that
\[
R\leq S(m,\Omega)+o(y/\ln y)\leq \frac{y}{3\ln y}.
\]
Since for the third step we have $\pi(y)-\pi(y/2)=\frac{y}{(2+o(1))\ln y}$ primes $p\leq y$ left, one can sift one of remaining numbers at a time to ensure that every number $i\leq m$ is sifted out in these three steps. This concludes the proof.
\end{proof}

\section{Computation of $M(f)$}

In this section, we prove Theorem 2 on the value of $M(f)$. The quantity $M(f)$ is the logarithmic average of $M_p(f)$ over prime values of $p$. The number $M_p(f)$ is defined as the maximal number of solutions for $f(x)=y$ with $x\in \mathbb F_p$ for a fixed non-zero $y\in \mathbb F_p$. We will show that $M(f)$ always exists and is a rational number and we also reduce its computation to group actions.

In the introduction, we associated $f(x)$ with several Galois groups, which have a natural action on the set $X$ of all roots of the equation $f(x)-y$, where $y$ is a formal variable. Namely, $G_f$ is the Galois group of $f(x)-y$ over $\mathbb Q(y)$ and $G_f^+$ is the Galois group of the same polynomial over $K(y)$. It is easy to see that $G_f^+$ is a normal subgroup of $G_f$. Indeed, it is clear that $K$ is a finite normal extension of $\mathbb Q$. If $G_f^0$ is its Galois group, then there is a natural surjective homomorphism $\pi:G_f\to G_f^0$. More precisely, every automorphism of the splitting field of $f(x)-y$ over $\mathbb Q(y)$ sends constants to constants. Such an automorphism is an automorphism over $K(y)$ if and only if all the constant elements are fixed, hence $G_f^+=\ker \pi$ is a normal subgroup.

Birch and Swinnerton-Dyer \cite{BSD} were able to reduce the question on the number of solution of $f(x)=y$ to certain other Galois groups.

\begin{lemma}
Let $G_p^+$ be the Galois group of $f(x)-y$ over $\overline{\mathbb F}_p(y)$ and $G_p$ be the Galois group of the same polynomial over $\mathbb F_p(y)$. For $n\leq d=\deg f$ let $m_n$ be the number of $G_p$-invariant orbits of action of $G_p^+$ on ordered $n$-tuples of distinct roots of $f(x)=y$. Then for $m_n=0$ the system $f(x_1)=f(x_2)=\ldots=f(x_n)$ has no solutions in $\mathbb F_p$ with $x_i\neq x_j$ for $i\neq j$, and for $m_n>0$ the number of solutions is \[m_np+O(\sqrt p).
\]
\end{lemma}

For convenience, let us formulate Chebotarev's density theorem. First, define the Frobenius element.
\begin{definition}
Let $K\slash \mathbb Q$ be a finite Galois extenstion, $\mathcal O_K$ its ring of integers and $G_K$ --- its Galois group. Suppose that $p$ is an unramified rational prime number and $\mathfrak p\subset \mathcal O_K$ is a prime ideal lying over $p$. The Frobenius element of $\mathfrak p$ is an element $\sigma_{\mathfrak p}\in G_K$ with
\[
\sigma_{\mathfrak p}x\equiv x^p \pmod {\mathfrak p}
\]
for all $x\in \mathcal O_K$.

For different prime ideals $\mathfrak p$ lying over fixed prime number $p$ Frobenius elements are conjugated, therefore up to conjugation we can define the Frobienus element $\sigma_p$.
\end{definition}
Chebotarev's density theorem states the following
\begin{lemma}
Let $K\slash \mathbb Q$ be a finite Galois extension with Galois group $G_K$. Suppose that $C\subset G_K$ is a conjugacy class and $\pi(x;K,C)$ is the number of primes $p\leq x$ such that the Frobenius element $\sigma_p$ lies in $C$. Then for some $c_K>0$ we have
\[
\pi(x;G_K,C)=\frac{|C|}{|G_K|}li(x)+O(xe^{-c_K\sqrt{\ln x}}).
\]
\end{lemma}
\begin{proof}
See \cite{LaOd}.
\end{proof}

Let us now prove Theorem 2.
\begin{proof}[Proof of Theorem 2]
In \cite{BSD} it is proved that for large enough $p$ the group $G_p^+$ can be identified with the Galois group of $f(x)-y$ over $\mathbb C(y)$. It is easy to see that this group coincides with $G_f^+$. Next, large primes $p$ are unramified in $K$, so we fix some prime ideal $\mathfrak p$ lying over $p$ in $\mathcal O_K$ and consider the corresponding Frobenius element $\sigma_{\mathfrak p}$. The largest constant subfield of $(\mathbb F_p)_f$ can be identified with $\mathcal O_K\slash \mathfrak p$. Moreover, $G_p$ acts on constants via Frobenius automorphism and its powers. For large $p$, the group $G_p$ is generated by $G_f^+$ and the preimage of $\sigma_{\mathfrak p}$, i.e.
\[
G_p=\langle \pi_p^{-1}(\sigma_{\mathfrak p}),G_f^+ \rangle,
\]
where $\pi_p$ is the homomorphism of restriction of $G_p-$action on constants (its image is Galois group of $\mathcal O_K\slash \mathfrak p$ over $\mathbb F_p$), which allows us to consider $G_p$ as a subgroup of $G_f$. Let $g$ be any element of $\pi_p^{-1}(\sigma_{\mathfrak p})$. Let us show that in this case the maximal $n$ with $m_n>0$ is equal to $\max_{h\in G_f^+}|X^{hg}|$. Indeed, $G_f^+$ is normal in $G_p$, hence the latter group is generated by $g$ and $G_f^+$, so a $G_f^+-$orbit is $G_p-$invariant if and only if it is $g$-invariant. This means that if our orbit takes form $G_f^+(x_1,\ldots, x_n)$, then for any $h_1\in G_f^+$ there must exist $h_2\in G_f^+$ such that
\[
gh_2(x_1,\ldots,x_n)=h_1(x_1,\ldots,x_n).
\]
The subgroup $G_f^+$ is normal, so there exists  $h_3$ with $gh_2=h_3g$. Therefore, our condition is equivalent to the following: for any $h_3$ one can find $h_1$ such that $x_1,\ldots,x_n$ are fixed by $hg$, where $h=h_1^{-1}h_3$. This proves the desired identity. Thus, for large $p$ the number $M_p(f)$ is equal to $\max_{h\in G_f^+}|X^{hg}|$, where $g\in \pi_p^{-1}(\sigma_{\mathfrak p})\subset \pi_p^{-1}(\sigma_p)=:C$ and $\sigma_p$ is the conjugacy class of $\sigma_{\mathfrak p}$ in $G_f^0$. Since $G_f^+$ is normal, this number does not depend on the choice of $g$ and $\mathfrak p$. Chebotarev's density theorem shows that
\[
\sum_{p\leq x}M_p(f)=\sum_{C, g_0\in C}\frac{|C|}{|G_f^0|}\max_{h\in G_f^+}|X^{h\pi_p^{-1}(g_0)}|li(x)+O(xe^{-c_K\sqrt{\ln x}})=
\]
\[
li(x)\frac{1}{|G_f^0|}\sum_{g_0\in G_f^0}\max_{h\in G_f^+}|X^{h\pi_p^{-1}(g_0)}|+O(xe^{-c_K\sqrt{\ln x}}),
\]
where $C$ in the first sum runs over all conjugacy classes of $G_f^0$.
Next, $G_f^0=G_f/G_f^+$ and the value $\max_h |X^{hg}|$ depends only on the class of $g\in G_f$ in the quotient $G_f/G_f^+$. This means that the last sum is equal to $m(G_f,G_f^+;X)$. Therefore,
\[
\sum_{p\leq x}M_p(f)=m(G_f,G_f^+;X)li(x)+O(xe^{-c_K\sqrt{\ln x}})
\]
and by partial summation we obtain
\[
\sum_{p\leq x}\frac{M_p(f)}{p}=m(G_f,G_f^+;X)\ln\ln x+O(1),
\]
which concludes the proof.
\end{proof}

\begin{remark}
For a few natural examples, the number $m(G,H;X)$ turns out to be integer, at least when $H$ is normal in $G$ and actions of both $G$ and $H$ are faithful and transitive. However, there are examples of non-integrality: the group $G=S_4\times \mathbb Z/2\mathbb Z$ acts transitively on a set $X$ with $6$ elements. Taking $H=A_4$, we obtain $m(G,H;X)=\frac{7}{2}$. This example was found by Peter Kosenko, who also managed to find some examples of non-integrality for $|X|=8,9$ and $10$.
Incidentally, $S_4\times \mathbb Z/2\mathbb Z$ is the Galois group of generic even sextic polynomial. Unfortunately, we were not able to produce examples of polynomials $f(x)$ of degree $6$ with $G_f=S_4\times \mathbb Z/2\mathbb Z$ and $G_f^+=A_4$.
\end{remark}

Since now we have a tool for computation of $M(f)$, let us find $M(f)$ for $f_d$ and for a typical polynomial of degree $d$. We call a polynomial of degree $d$ \emph{generic} if $G_f^+\cong S_d$.
\begin{theorem}
Let $d$ be an integer. 
\begin{itemize}
    \item[i)] We have $M(x^d)=\tau(d)$, where $\tau(d)$ is the divisor function.
    \item[ii)] For any generic polynomial $f(x)$ of degree $d$ we have $M(f)=d$.
\end{itemize}
\end{theorem}

\begin{proof}
To prove the first part, note that for $f(x)=x^d$ we have 
\[M_p(f)=(p-1,d).
\]Indeed, in this case it is easy to see that if $y\neq 0$ and $x^d=y$ has at least one solution, then the number of solutions is equal to the number or $d-$th roots of unity in $\mathbb F_p$. Next, using the identity
\[
 (d,p-1)=\sum_{d_1\mid (d,p-1)}\varphi(d_1),
 \]
 we get for any $x\geq 2$
\[
\sum_{p\leq x}\frac{M_p(f_d)}{p}=\sum_{p\leq x}\frac{(d,p-1)}{p}=
\]
\[
=\sum_{p\leq x}\frac{\sum_{d_1\mid (d,p-1)}\varphi(d_1)}{p}=\sum_{d_1\mid d}\varphi(d_1)\sum_{\substack{p\leq x\\ p\equiv 1\pmod {d_1}}}\frac{1}{p}.
\]
Next, from Prime Number Theorem for primes in arithmetic progressions we find
\[
\sum_{\substack{p\leq x\\ p\equiv 1\pmod {d_1}}}\frac{1}{p}=\frac{\ln\ln x}{\varphi(d_1)}+O(1).
\]
Substituting this identity into the formula above, we get the desired identity: the factors $\varphi(d_1)$ and $\frac{1}{\varphi(d_1)}$ cancel out and we are left with the sum $\sum_{d_1\mid d}1=\tau(d)$. This proof does not rely on Theorem 2, but can be easily reformulated in terms of the natural action of affine group $\mathrm{Aff}(\mathbb Z\slash d\mathbb Z)$ on $\mathbb Z\slash d\mathbb Z$. Indeed, in our case $K=\mathbb Q(\zeta_d)$ is the cyclotomic field. Every automorphism of $\mathbb Q_f$ sends $\zeta_d$ to $\zeta_d^a$ for some $a\in \left(\mathbb Z\slash d\mathbb Z\right)^*$ and $y$ to $\zeta_d^b y$ for some $b\in \mathbb Z\slash d\mathbb Z$. The map $(a,b)\mapsto (x\mapsto ax+b)$ is an isomorphism $G_f\cong \mathrm{Aff}(\mathbb Z\slash d\mathbb Z)$. Since $G_f^+$ fixes $\zeta_d$, we also see that $G_f^+$ corresponds to the shift subgroup $y\mapsto y+b$ and is isomorphic to $\mathbb Z\slash d\mathbb Z$.

As for generic polynomials, we have $G_f^+\cong S_d$, hence $G_f=G_f^+$ and for all $g\in G_f$
\[
\max_{h\in G_f^+}|X^{gh}|=|X^e|=d,
\]
as needed.
\end{proof}

The paper \cite{BSD} contains, among other things, some criteria for $f(x)$ to be generic. In particular, authors prove that if $f'(x)$ has simple roots and $f(\alpha)\neq f(\beta)$ for any two distinct roots, then $f(x)$ is generic. Of course, this means that coefficients of any non-generic polynomial must satisfy certain polynomial relation. This, in turn, implies that most polynomials of degree $d$ are generic.

It is also clear that our proof of $M(f)=d$ for generic polynomials uses only $G_f=G_f^+$. Obviously, the latter condition is actually equivalent to $M(f)=d$. Since $G_f^0=G_f\slash G_f^+$, it is also equivalent to $K=\mathbb Q$, i.e. all the constants in $\mathbb Q_f$ must be rational. Let us call such polynomials \emph{ordinary}.

Currently, we are not aware of any simple way to test whether a given polynomial is ordinary, besides more narrow criteria of Birch and Swinnerton-Dyer.

The polynomials that we call generic are called Morse functions by J.-P. Serre in his book on Galois theory \cite[Section 4.4]{Ser}. It is also mentioned that Hilbert established an isomorphism $G_f\cong S_n$ for all such polynomials, see also \cite{Hilb}

Similarly to the results on gaps between primes, our results imply the following
\begin{corollary}
For any $f(x)\in \mathbb Z[x]$ there exists a constant $c_f>0$ such that for all $X\geq e^{16}$ there is a positive integer $n\leq X$ with $\Omega(n+f(m))>A(n)$ for all $1\leq m\leq Y$, where
\[
Y=c_f\ln X(\ln\ln X)^{\ell_f-1}\left(\frac{(\ln\ln\ln X)^2}{\ln\ln\ln\ln X}\right)^{h_f}\left(\frac{\ln\ln X\ln\ln\ln\ln X}{(\ln\ln\ln X)^2}\right)^{M(f)},
\]
$\Omega$ is the number of prime factors counted with multiplicity and $A(n)$ is the number of irreducible factors of $f(x)+n$.
\end{corollary}
One can also note that our proofs can be extended to integer-valued polynomials $f(x)$ which are not necessarily in $\mathbb Z[x]$.

\section{Acknowledgements}

We would like to thank Peter Kosenko for providing an abstract example of non-integrality of $m(G,H;X)$ (see Remark 2) and Mathoverflow user Joachim K\"{o}nig for providing the main idea for the proof of Theorem 2: \href{https://mathoverflow.net/a/430001/101078}{https://mathoverflow.net/a/430001/101078}. We also thank anonymous referees for correcting several inaccuracies, especially in Section 3, and for pointing out the references to Serre and Hilbert.\\This work was performed at the Steklov International Mathematical Center and supported by the Ministry of Science and Higher Education of the Russian Federation (agreement no. 075-15-2022-265). The first author was supported by the Basic Research Program of the National Research University Higher School of Economics.

\newpage
 \bibliographystyle{amsplain}

\end{document}